\documentclass[11pt,reqno]{amsart}
\usepackage{graphicx,psfrag,amsxtra,color,url,scalerel,amsthm,xspace}
\usepackage{amscd}
\usepackage{multirow}
\usepackage{xy}
\xyoption{all}
\usepackage{pinlabel}
\usepackage[linktocpage]{hyperref}

\usepackage{bm}
\usepackage[left=1.5in,top=1.2in,right=1.5in,bottom=1.2in,head=.2in]{geometry}
\usepackage{mathabx}

\newcommand{\FF}{\mathbb F}
\newcommand{\Z}{\mathbb Z}
\newcommand{\Q}{\mathbb Q}
\newcommand{\R}{\mathbb R}
\newcommand{\C}{\mathbb C}

\newcommand{\im}{\operatorname{im}}

\newcommand{\Hom}{\operatorname{Hom}}
\newcommand{\rank}{\operatorname{rank}}
\newcommand{\Ker}{\operatorname{Ker}}

\newcommand{\inte}{\operatorname{int}}
\newcommand{\zd}{\Z_d}
\newcommand{\zp}{\Z_p}

\newcommand{\rpone}{\R\textup{P}^1}
\newcommand{\rptwo}{\R\textup{P}^2}

\newcommand{\cptwo}{\C\textup{P}^2}
\newcommand{\cpone}{\C\textup{P}^1}
\newcommand{\cptwobar}{\overline{\C\textup{P}}\,\!^2}
\def\conn{\mathbin{\#}}

\newcommand{\xtilde}{\widetilde{X}}

\newcommand{\vtilde}{\widetilde{V}}
\newcommand{\ftilde}{\widetilde{F}}

\newcommand{\A}{\mathcal A}
\newcommand{\B}{\,\mathcal B}

\newcommand{\F}{\mathcal F}

\renewcommand{\P}{\mathcal P}

\newcommand{\V}{\mathcal V}

\renewcommand{\phi}{\varphi}

\newcommand{\sign}{\operatorname{sign}}

\newcommand{\spinc}{\ifmmode{\operatorname{Spin}^c}\else{$\operatorname{spin}^c$\ }\fi}

\newcommand{\cpl}{$\C$-pseudoline\xspace}
\newcommand{\rpl}{$\R$-pseudoline\xspace}

\newcommand{\lc}{\mathcal L} 
\newcommand{\Lc}{\mathcal L}

\DeclareMathOperator*{\Bigcdot}{\scalerel*{\cdot}{\bigodot}}

\newtheorem{theorem}{Theorem}[section]

\newtheorem{lemma}[theorem]{Lemma}
\newtheorem{proposition}[theorem]{Proposition}
\newtheorem{corollary}[theorem]{Corollary}

\theoremstyle{definition}
\newtheorem{definition}[theorem]{Definition}
\newtheorem{remark}[theorem]{Remark}

\newtheorem*{T:non-filling}{Theorem \ref{T:non-filling}}

\def\acknowledgementname{Acknowledgements.}
\newenvironment{acknowledgement}
   {
    \vspace*{\baselineskip}
    \noindent{\small\textbf{\acknowledgementname}}
    \unskip\noindent}{}
\title{Topological realizations of line arrangements}
\author[Daniel Ruberman]{Daniel Ruberman${}^1$}
\address{Department of Mathematics, MS 050\newline\indent Brandeis
University \newline\indent Waltham, MA 02454}
\email{ruberman@brandeis.edu}
\author[Laura Starkston]{Laura Starkston${}^2$}
\address{Department of Mathematics, \newline\indent  
Stanford University \newline\indent Stanford, CA 94305}
\email{lstarkst@stanford.edu}
\thanks{\noindent 
${}^1$\today. Partially supported by NSF Grant \#1506328 and NSF FRG Grant \#1065827
\noindent 
${}^2 $Partially supported by an NSF Postdoctoral Research Fellowship, Grant \#1501728
}
\begin{document}
\begin{abstract}
A venerable problem in combinatorics and geometry asks whether a given incidence relation may be realized by a configuration of points and lines. The classic version of this would ask for lines in a projective plane over a field. An important variation allows for {\em pseudolines}: embedded circles (isotopic to $\rpone$) in the real projective plane. In this paper we investigate whether a configuration is realized by a collection of $2$-spheres embedded, in symplectic, smooth, and topological categories, in the complex projective plane.  We find obstructions to the existence of topologically locally flat spheres realizing a configuration, and show for instance that the combinatorial configuration corresponding to the projective plane over any finite field is not realized. Such obstructions are used to show that a particular contact structure on certain graph manifolds is not (strongly) symplectically fillable. We also show that a configuration of real pseudolines can be complexified to give a configuration of smooth, indeed symplectically embedded, $2$-spheres.
\end{abstract}
\maketitle

\vspace{-.2in}

\section{Introduction}
A (complex projective) line arrangement is a collection of distinct projective lines in $\cptwo$. Any two lines intersect at a single point, but there may be multiple points, where more than two lines meet. Abstracting this property leads to the notion of \emph{combinatorial arrangements}: incidence relations which have a unique point on each pair of lines. A combinatorial arrangement is complex geometrically realizable if there is a configuration of complex lines in $\cptwo$ with the specified incidences. Moreover one can study the moduli space of geometric realizations, which is an algebraic variety. While line arrangements are simple objects to define, these moduli spaces can be very complicated. If one generalizes from line arrangements to hyperplane arrangements, such moduli spaces can realize any algebraic variety by Mnev's universality theorem \cite{mnev:universality}. Real and complex line arrangements have been studied for over a century, though many questions remain unanswered.

In this paper, we consider other types of topological and geometric realizations of combinatorial line arrangements in $\cptwo$. From a topological viewpoint, a complex projective line arrangement is a configuration of transversally embedded $2$-spheres, with a given pattern of intersections and such that each sphere is homologous to $\cpone$. We focus on the cases when such spheres are symplectic, smooth, or topologically locally flat. These categories are much less rigid than than the algebro-geometric one, so the moduli spaces of such realizations are no longer finite dimensional varieties; locally there is an infinite dimensional space of perturbations of a sphere in these categories. Despite the significant local differences, one can ask if there are topological differences between the moduli space of realizations in different categories for a fixed combinatorial arrangement. The most basic topological property is whether the moduli space is empty. In principle, a combinatorial arrangement may be realized in one category (say, smoothly) but not in another one (say, as algebro-geometrically).

A strong motivation for this course of study is to understand the symplectic case. Gromov's theory of pseudoholomorphic curves has allowed symplectic geometers to adapt many techniques from complex geometry. In particular, symplectic surfaces in $\cptwo$ can often be treated as complex curves, and Gromov used this idea to show that smooth symplectic surfaces in homological degrees $1$ and $2$ are all symplectically isotopic to complex algebraic curves \cite{gromov}. Extending this result to other symplectic surfaces (either smooth or with prescribed singularities) in $\cptwo$ is known as the \emph{symplectic isotopy problem}. The symplectic isotopy problem has been solved for some rather low degree cases by Sikorav, Shevchishin, and Siebert and Tian \cite{shevchishin,sikorav,sieberttian,sieberttian:symplecticisotopy}. For smooth surfaces these results cover up to degree $17$, and there are also limited results for low degree curves with certain prescribed singularities. Symplectic line arrangements provide a class of singular curves on which this problem can be tested. The second author showed in \cite{starkston:comparing} that for certain simple combinatorial arrangements, the moduli space of symplectic line arrangements is path connected, non-empty, and contains complex solutions. In this paper we study symplectic realizability of more complicated combinatorial arrangements, particularly those which cannot be realized algebro-geometrically.

In fact there are a few examples of differences between the algebraic category and the more flexible symplectic, smooth, or topological categories, coming from theorems in complex projective geometry. For example, Pappus' theorem states that in a configuration of lines intersecting as in figure \ref{fig:Pappus}, a line through any two of the points $P$, $Q$, and $R$ necessarily passes through the third. Thus an arrangement which contains a line passing through $P$ and $Q$ but not $R$ is not algebro-geometrically realizable, but it can be obtained from the algebro-geometric arrangement by slightly perturbing the line near $R$ (which is allowable in the topological, smooth, and symplectic categories). Such examples show that there are some differences between realizability in some of the categories. 

For each combinatorial arrangement the complex realization space has an expected complex dimension which increases by two for each line and decreases for each multiple point where more than two lines intersect. We think of combinatorial arrangements with smaller expected dimension as more degenerate configurations. In the Pappus example, the combinatorial arrangement which is symplectically realizable is less degenerate than the one which is complex algebraically realizable. One might hope for the opposite phenomenon: that working in the more flexible topological, smooth, and symplectic categories would allow one to realize combinatorial arrangements which are too degenerate to be algebro-geometrically realized (where the realization space is empty and has negative expected dimension). A priori, one might even expect that in the topological category any combinatorial arrangement can be realized. We show that in fact this is not the case.
\begin{figure}
	\centering
	\includegraphics[scale=.6]{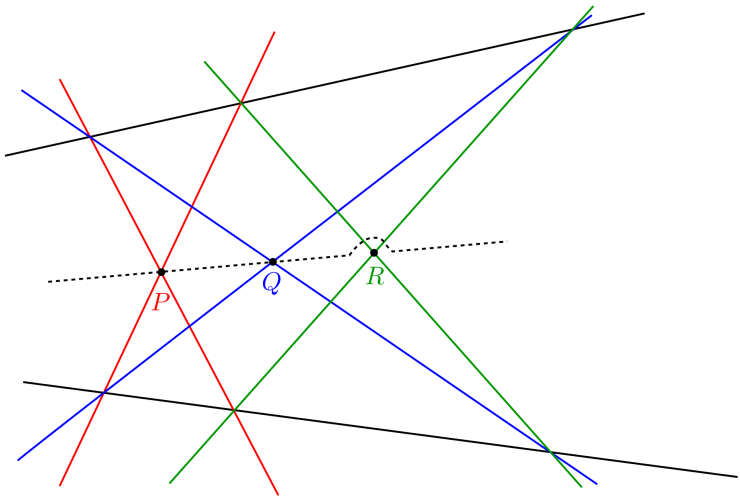}
	\caption{Pappus' theorem. A configuration of lines including the dotted one which passes through $P$ and $Q$ but not $R$ can be realized symplectically but not complex algebraically.}
	\label{fig:Pappus}
\end{figure}

In this paper, we give {\em topological} obstructions to realizing a combinatorial arrangement by a {\em topological} configuration. We show that a family of combinatorial arrangements, given by the projective planes over finite fields of order $p^d$, which is not realized by a complex line arrangement,  is not even realized topologically. It has been known since the 1800s that these combinatorial arrangements are not complex algebro-geometrically realizable; such a realization would require a solution to a system of equations that has solutions only when $p=0$.
\begin{theorem}\label{T:projplane}
Let $p$ be a prime and $d=p^q$ for some $q\geq 1$. There is no topologically flat embedding of transverse spheres of self-intersection $1$ in $\cptwo$ whose intersections give the incidence matrix of the projective plane over $\FF_{d}$.
\end{theorem}

There has been considerable study of configurations of points and lines in the real (projective) plane. Much of our knowledge and terminology on configurations here comes from the book of Gr{\"u}nbaum \cite{grunbaum:configurations}, which contains many fascinating results about real geometric and topological realizations. An \emph{$(n_k)$-configuration} is a line arrangement consisting of $n$ lines with $n$ distinguished points of multiplicity $k$ such that each line contains exactly $k$ distinguished points. The finite projective plane arrangements obstructed by Theorem~\ref{T:projplane} are $(n_k)$-configurations for $n=d^2+d+1$ and $k=d+1$. In this case, (and generally whenever $n=k(k-1)+1$) all of the points of the configuration are distinguished, so in some sense, these are the most degenerate types of configurations. Increasing $n$ while holding $k$ fixed yields configurations which gradually become easier to realize. It is known that when $n=k^k$, there exists a realization of an $(n_k)$-configuration by straight lines in the real plane. It is generally difficult to understand the realizability of the configurations for $k(k-1)+1\leq n <k^k$ in the various categories in the real and complex projective planes, though some results are known in the real case for $k=3,4$.

For example, the unique combinatorial $(8_3)$-configuration cannot be realized in the real projective plane (see~\cite[\S 2.1]{grunbaum:configurations} for a proof), but it can be realized geometrically in the complex projective plane \cite{levi:configs}. The configuration $(8_3)$ is the ``second most degenerate'' $(n_3)$ configuration to consider after the projective plane of order $2$ (the Fano plane) and already it becomes realizable in $\cptwo$. However if we consider $(n_4)$ configurations, we are able to obstruct $\C$-realizability further. The projective plane of order $3$, obstructed by Theorem~\ref{T:projplane}, is the unique most degenerate $(13_4)$ configuration. There is also a unique combinatorial $(14_4)$-configuration, and in this case we are able to obstruct it topologically.
\begin{theorem}\label{T:144config}
	There is no topological $\C$-realization of a $(14_4)$ configuration.
\end{theorem}

In addition to obstructing topological (and thus smooth and symplectic) realizations, we also give a method of constructing symplectic line arrangements from a configuration of topologically embedded lines in $\rptwo$ (known as {\em pseudolines} in the literature~\cite{grunbaum:configurations}).  To be clear about whether we are discussing real or complex realizability, we will refer to a topologically embedded configuration in $\rptwo$ as an \emph{$\R$-pseudoline arrangement} or \emph{topological $\R$-realization}, and the corresponding object in $\cptwo$ will be called a \emph{$\C$-pseudoline arrangement} or a \emph{topological/smooth/symplectic/geometric $\C$-realization}. 

There is a natural embedding of $\rptwo$ in $\cptwo$ as the points with real coordinates.  Any arrangement of real lines in $\rptwo$ is given by a collection of real linear equations; the set of complex solutions to these equations is an arrangement of complex lines with the same incidences.  We show an analogous fact for pseudolines.
\begin{theorem}\label{T:real-complex}
Let $\lc_\R$ be an $\R$-pseudoline configuration in $\rptwo$. Then there is a configuration $\lc_\C$ of smoothly embedded $2$-spheres in $\cptwo$ whose intersection with $\rptwo$ is $\lc_\R$, such that the intersections amongst the spheres in $\lc_\C$ are precisely those amongst the pseudolines in $\lc_\R$. The configuration $\lc_\C$ is invariant under complex conjugation. 
\end{theorem}

Theorem \ref{T:real-complex} shows that a topological realization in the real projective plane gives rise to a topological realization in the complex projective plane. The converse does not hold though; the $(8_3)$-configuration shows that obstructions to $\R$-pseudoline arrangements do not generally extend to obstruct $\C$-pseudoline arrangements. In addition to a purely topological complexification, we are able to build symplectic configurations from an \rpl\ arrangement.

\begin{theorem}\label{T:real-symplectic}
	Let $\lc_\R$ be an $\R$-pseudoline configuration in $\rptwo$. Then $\lc_\R$ is isotopic to a configuration with a complexification as in theorem \ref{T:real-complex} consisting of symplectic spheres.
\end{theorem}

Constructing \rpl\ arrangements is very concrete since they can be drawn in $\R^2$, and there are known differences between real straight line arrangements and \rpl\ arrangements. Sufficiently deep examples of combinatorial line arrangements with topological $\R$-realizations, but no algebro-geometric $\C$-realization could yield constructions of symplectic 4-manifolds which exhibit properties that complex surfaces cannot have through branched covering or surgery constructions. Optimistically, one might even hope that such examples could yield symplectic counterexamples to the Bogomolov-Miyaoka-Yau inequality.

Another symplectic application of this exploration yields results about symplectic fillings of certain contact $3$-manifolds. Many of the known symplectic filling classifications eventually reduce to a classification of symplectic curves in $\cptwo$ with prescribed singularities (see \cite{mcduff:rationalruled, lisca, ohtaono:symplecticfillings, ohtaono:simplesingularities, starkston:fillings, gollalisca:stein}). The line arrangement case is relevant to fillings of a large class of Seifert fibered spaces in \cite{starkston:fillings}, and here we give another application. To a combinatorial arrangement $\A$, we associate a canonical contact structure on a $3$-manifold $Y_\A$. Roughly speaking, $Y_\A$ is the boundary of a neighborhood of $2$-spheres intersecting according to $\A$, built through a plumbing type construction (see section \ref{S:non-fillable} for the precise definition). The manifolds $Y_\A$ are graph manifolds, but typically have large first Betti number. We show non-realizability implies non-fillability.
\begin{theorem}\label{T:non-filling}
	Suppose a combinatorial line arrangement $\A$ is not symplectically realizable in $\cptwo$. Then the canonical contact manifold $Y_\A$ is not strongly symplectically fillable.
\end{theorem}

\begin{acknowledgement}
We thank Daniel Bump, Pat Gilmer, Kiyoshi Igusa, Tye Lidman, Patrick Massot, Nathan Reading, Vivek Shende, Ivan Smith, and Alex Suciu for helpful discussions and correspondence on the material in this paper.
This project was initiated at the Workshop {\em Geometry and Topology of Symplectic 4-manifolds} at the University of Massachusetts, April 2015. 
\end{acknowledgement}

\section{Combinatorial line arrangements}
\subsection{Definitions}

\begin{definition}
	A \emph{combinatorial line arrangement} $\mathcal{A}$ consists of a set of lines $\mathcal{A}_{\mathcal{L}}=\{L_1,\cdots, L_\ell\}$, a set of points $\mathcal{A}_{\mathcal{P}}=\{P_1,\cdots , P_p\}$, and an incidence relation. The incidence relation is a pairing $\cdot: \mathcal{A}_{\mathcal{L}}\times \mathcal{A}_{\mathcal{P}} \to \{0,1\}$ such that every pair of lines intersects at a unique point, meaning for every pair $L_i\neq L_j$, there is a unique point $P_k$ such that $L_i\cdot P_k=L_j\cdot P_k=1$.
\end{definition}

We can encode the data of a combinatorial line arrangement in a matrix as follows.
\begin{definition}
	An \emph{incidence matrix} for a combinatorial line arrangement $\A$ is an $|\A_\Lc|\times|\A_\P|$ matrix whose entry in the $i^{th}$ row and $j^{th}$ column is $L_i\cdot P_j$.
\end{definition}

\begin{definition}
	Given a point $P\in \A_\P$, we say its \emph{multiplicity} is the number of lines containing $P$. If $P$ has multiplicity $2$ we say $P$ is a \emph{double point}. If $P$ has multiplicity $\geq 3$ we say $P$ is a \emph{multi-point}.
\end{definition}

\begin{definition}
	A \emph{geometric/symplectic/smooth/topological $\C$-realization} of a combinatorial line arrangement $\A$ is a collection of $|\A_\Lc|$ algebro-geometrically/symplectically/ smoothly/topologically embedded 2-spheres in $\cptwo$, each in the homology class of the complex projective line, whose intersections are all positive and transverse and are specified by the incidence relations of $\A$.	
	Similarly a \emph{geometric/smooth/topological $\R$-realization} of a combinatorial line arrangement $\A$ is a collection of $|\A_\Lc|$ algebro-geometrically/smoothly/topologically embedded circles in $\rptwo$, each in the homology class of the real projective line, with intersections specified by $\A$.
\end{definition}

\begin{definition}
	A combinatorial line arrangement $\mathcal{A}'$ is a \emph{sub-arrangement} of $\mathcal{A}$ if there is an incidence preserving injection 
	$$(\phi_\Lc,\phi_\P): \mathcal{A}'_{\mathcal{L}}\times \mathcal{A}'_{\mathcal{P}}\to \mathcal{A}_{\mathcal{L}}\times \mathcal{A}_{\mathcal{P}}.$$
	We say $\A'$ is a \emph{strict sub-arrangement} of $\A$ if it has the additional property that if $P\in \mathcal{A}_{\mathcal{P}}$ and there is some $L\in \mathcal{A}'_{\mathcal{L}}$ such that $P\cdot \phi_\Lc (L)=1$  then $P\in \im(\phi_\P)$. 
	
	Intuitively, in a strict sub-arrangement all the points of $\mathcal{A}$ which are on lines of $\mathcal{A}'$ are points of $\mathcal{A}'$.
\end{definition}

Clearly, if $\A'$ is not realizable in some category and $\A'$ is a sub-arrangement of $\A$, then $\A$ is not realizable in that category.

\subsection{Examples of combinatorial line arrangements}\label{sec:examples}

\begin{figure}[htb]
	\labellist
	\small\hair 2pt
	\pinlabel {$S_1$} [ ] at 67 133
	\pinlabel {$S_2$} [ ] at 99 133
	\pinlabel {$S_3$} [ ] at 117 133
	\pinlabel {$S_4$} [ ] at 153 28
	\pinlabel {$S_5$} [ ] at 140 10
	\pinlabel {$S_6$} [ ] at 25 25
	\pinlabel {$S_7$} [ ] at 66 18
	\pinlabel {$Q_1$} [ ] at 93 170
	\pinlabel {$Q_2$} [ ] at 150 87
	\pinlabel {$Q_3$} [ ] at 193 03
	\pinlabel {$Q_4$} [ ] at 93 -8
	\pinlabel {$Q_5$} [ ] at -10 03
	\pinlabel {$Q_6$} [ ] at 34 87
	\pinlabel {$Q_7$} [ ] at 105 55
	\endlabellist
	\centering
	\includegraphics[scale=1]{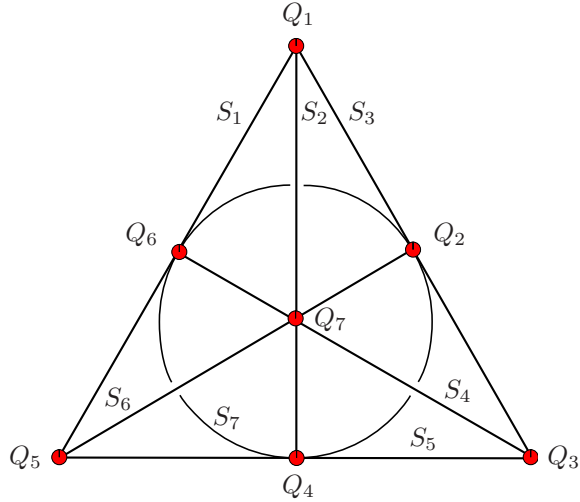}
	\caption{The Fano plane}
	\label{F:fano}
\end{figure}

A natural class of combinatorial line arrangements arise from the points and lines in the projective plane over a finite field $\FF_{p^q}$ for $p$ prime, which we will denote by $P(2,p^q)$. We will find obstructions to realizations of these combinatorial line arrangements. First we notice that the combinatorial line arrangement coming from $P(2,p^q)$ contains the combinatorial line arrangement coming from $P(2,p)$ as a (non-strict) sub-arrangement. We can see this by observing that the points of $P(2,p)$ in homogeneous coordinates $\{[x:y:z]| x,y,z\in \FF_p, (x,y,z)\neq (0,0,0)\}$ are a subset of the points of $P(2,p^q)$ in homogeneous coordinates $\{[x:y:z]| x,y,z\in \FF_{p^d}, (x,y,z)\neq (0,0,0)\}$ because $\FF_{p^q}$ is a field extension of $\FF_p$. Similarly the linear homogeneous equations with coefficients in $\FF_p$ which form the collection of lines in $P(2,p)$ also represent lines in $P(2,p^q)$, with the same incidences to points of $P(2,p)$. Therefore, it suffices to obstruct realizability for the projective planes $P(2,p)$. From now on, we will restrict to discussing this case.

$P(2,p)$ consists of $p^2+p+1$ points and $p^2+p+1$ lines. The points and lines are dual to each other. The incidence data in the case $p=2$ is schematically represented by figure \ref{F:fano} which is often referred to as the \emph{Fano plane}. The projective planes of order $p$ have the symmetry that every line contains $p+1$ points of multiplicity $p+1$. Generalizing this symmetry leads to the notion of an \emph{$(n_k)$-configuration}.
\begin{definition}[\cite{grunbaum:configurations}]
	An $(n_k)$-configuration is a combinatorial line arrangement $\A$ with $n$ lines and a subset of distinguished points $S\subset \A_\P$ with $|S|=n$ such that each point $p\in S$ has multiplicity $k$ and each line $L\in \A_\Lc$ contains $k$ of the distinguished points.
\end{definition}

Often the term configuration refers to a geometric or topological realization of such combinatorics, but we will use this terminology to refer to the combinatorial configuration so that we do not need to a priori know whether that combinatorics is realizable to name the combinatorial object. Observe that $P(2,p)$ is an $(n_k)$-configuration where $n=p^2+p+1$ and $k=p+1$. Note that in this case $n=k(k-1)+1$. In general, an $(n_k)$-configuration satisfies the inequality
$$n\geq k(k-1)+1$$
because each distinguished point has $k$ lines through it, and each of those lines contains $k-1$ other distinguished points. Note that fixing a particular line in a configuration, $L_0\in \A_{\Lc}$, the distinguished points account for the intersections of $L_0$ with $k(k-1)$ other lines. When $n=k(k-1)+1$ this accounts for all of the intersections of $L_0$ with any other line of the configuration but when $n>k(k-1)+1$ there are other points in the combinatorial line arrangement on $L_0$ which are not in the distinguished set. 

\section{Mod $d$ relations and branched covers}\label{S:relations}
Consider an arbitrary simply-connected $4$-manifold $X$, containing a collection $\F = \{F_1,\ldots,F_n\}$ of disjoint oriented surfaces of self-intersection $(-k)$; in this paper they will be spheres (often referred to as $(-k)$-spheres). We will assume $k \neq 0$, which, because the surfaces are disjoint, implies that the homology classes of the $F_i$ are linearly independent over $\Z$ (and $\Q$ for that matter). In particular, $n \leq m = \rank(H_2(X))$.  However, for $d$ dividing $k$, there may be a {\em mod $d$ relation} among the $F_i$ of the form
\begin{equation}\label{E:relation}
\sum_i a_i [F_{i}] = 0\quad \text{in } H_2(X;\Z_d)\quad\text{with } a_i \in \Z_d.
\end{equation}
Note that the set of mod $d$ relations is a $\Z_d$ module $\V$, in fact a submodule of $(\Z_d)^n = H_2(\F;\Z_d)$. Such objects are studied in coding theory. They can be interpreted in terms of cyclic branched coverings of $X$ of order $d$ with branch set a subset of $\F$, or $\Z_d$ coverings for short. Using the language of  linear codes, the {\em weight} of a relation is the number of nonzero $a_i$ appearing in the sum. The weight of the relation is highly significant in determining the topology of the branched cover because it determines the number of surfaces in the branch set. We will use the following notation: For each $i$, let $N_i$ denote the normal disk bundle of $F_i$, and let $N(\F)$ be the union of the $N_i$, with $F$ the union of the $F_i$. (We will just write $N$ when it is clear which $\F$ is being discussed.)

\begin{proposition}\label{P:covers}
 $\Z_d$ covers with branch set exactly $F$ are in one-to-one correspondence with mod $d$ linear relations in which every coefficient is non-zero.
\end{proposition}
\begin{proof}
Smooth $\Z_d$ covers with branch set contained in $F$ correspond to elements 
$$
\phi \in H^1(X - \inte(N);\zd) \cong \Hom(H_1(X - \inte(N)),\zd)
$$
where the branching is non-trivial along a component $F_i$ if and only if the evaluation of $\phi$ on its meridian is nonzero \cite[\S 2]{cassongordon}. Poincar\'e duality gives an isomorphism
$$
H^1(X - \inte(N);\zd) \cong H_3(X - \inte(N),\partial N;\zd).
$$
The dual of a class $\phi \in H^1(X - \inte(N);\zd)$ is represented by a relative $3$-cycle $K$ such that $\partial K \cap \partial N_i$ is the dual of the restriction of $\phi$ to $\partial N_i$.  By excision and the long exact sequence of $(X, F)$, we have
$$
 H_3(X - \inte(N),\partial N;\zd) \cong H_3(X,F;\zd) \cong \Ker\left[ H_2(F;\Z_d) \to H_2(X;\zd)\right].
 $$
Hence the relative homology class of $K$ corresponds uniquely to a relation $\sum_i a_i [F_{i}] = 0\in H_2(X;\zd)$ where the nonzero coefficients correspond to nontrivial branching.
\end{proof}
By applying this argument to various subsets of $\F$, we can obtain some additional information about the topology of the complement of $F$.  This is easiest to understand when $d$ is a prime $p$. Compare~\cite[Proposition 2.2]{ruberman:config}, where the case $p=2$ is discussed. We state a corollary of the above proof which we will need later.
\begin{corollary}\label{C:cyclic}
Suppose that there is a mod $p$ relation of the form~\eqref{E:relation} in which every $a_i$ is non-zero. Then
$H^1(X - F; \zp)$ has rank at least $2$ if and only if there is a mod $p$ relation of the form
$$
\sum_{F_i} a_i' [F_{i}] = 0\quad \text{in } H_2(X;\zp)
$$
with the vectors $(a_1,\ldots,a_n)$ and $(a_1',\ldots,a_n')$ linearly independent.
\end{corollary}

\subsection{Line arrangements and mod $d$ relations}\label{sec:linerels}

In this paper we will be focusing on mod $d$ relations amongst disjoint spheres in a blow-up of $\cptwo$ which arise as the proper transforms of lines in an arrangement after blowing up at the intersection points. Let $L_1,\cdots, L_n$ denote the lines of the geometric/symplectic/smooth/topological realization, and $F_1,\cdots, F_n$ their corresponding proper transforms after blowing up at $N$ points in the arrangement. The homology class of the proper transform $F_i$ is
$$[F_i]=h-\sum_{j=1}^N (L_i\cdot P_j)e_j$$
where $L_i\cdot P_j$ denotes the incidence (either $0$ or $1$) between the line and the point. In particular, if we have a mod $d$ relation as in equation \eqref{E:relation}, then looking at the coefficient of $h$ we find \begin{equation}\sum_{i=1}^\ell a_i \equiv 0 \mod d \label{eqn:plines}\end{equation}
and looking at the coefficients of the exceptional classes we find
\begin{equation}\sum_{i=1}^\ell a_i (L_i\cdot P_j) \equiv 0 \mod d
	\label{eqn:exceptional}
\end{equation} 
for every point $P_j$, $j=1,\cdots, N$. This can be interpreted in the original line arrangement by associating the weight $a_i$ to the line $L_i$, and asking that at each intersection point which we blow up at, the sum of the weights of the lines through that point should be $0$ mod $d$. Additionally, the total sum of the weights is be $0$ mod $d$. This geometric interpretation through the line arrangement can be useful in understanding the set of all mod $d$ relations for $F_1,\cdots, F_n$.

For example, if we blow up at all points in a hypothetical realization of the Fano plane $P(2,2)$ ($d=2$), $\mod d$ relations have coefficients $0$ or $1$. Therefore these relations are binary codewords corresponding to a subset of lines of the Fano plane in which  each vertex appears in an even number of lines. For example, the subarrangement in figure \ref{F:relation} corresponds to the relation $F_2+ F_3 + F_5 + F_6 \equiv 0 \pmod{2}$. Pictorially, we see that there is an even number of lines of the sub-arrangement through every point $Q_j$. Computationally the relation is easily verified:
\begin{equation*}
\begin{split}
F_2+ F_3 + F_5 + F_6 & =
(h - e_1 - e_4 - e_7) + (h - e_1 - e_2 - e_3)\\ &\qquad + (h - e_3 - e_4 - e_5) + (h - e_2 - e_5 - e_7)  \\
&  = 4h - 2(e_1 + e_2 + e_3 + e_4 + e_5 + e_7).
\end{split}
\end{equation*}
\begin{figure}[htb]
	\labellist
	\small\hair 2pt
	\pinlabel {$L_2$} [ ] at 99 133
	\pinlabel {$L_3$} [ ] at 117 133
	\pinlabel {$L_5$} [ ] at 140 10
	\pinlabel {$L_6$} [ ] at 25 25
	\pinlabel {$Q_1$} [ ] at 93 170
	\pinlabel {$Q_2$} [ ] at 150 87
	\pinlabel {$Q_3$} [ ] at 193 03
	\pinlabel {$Q_4$} [ ] at 93 -8
	\pinlabel {$Q_5$} [ ] at -10 03
	\pinlabel {$Q_6$} [ ] at 34 87
	\pinlabel {$Q_7$} [ ] at 105 55
	\endlabellist
	\centering
	\includegraphics[scale=1]{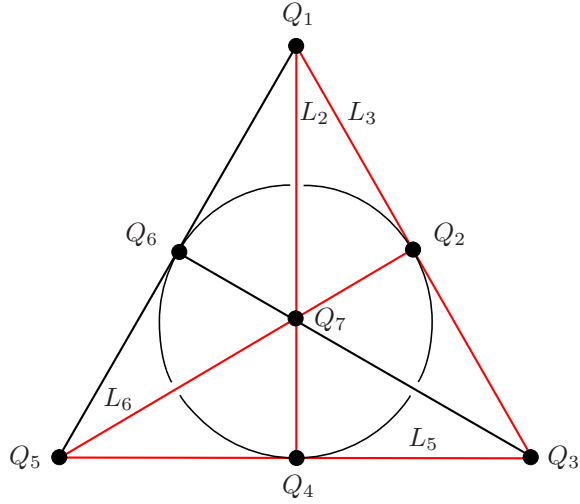}
	\caption{A mod $2$ relation}
	\label{F:relation}
\end{figure}

\subsection{The topology of branched covers}
For the remainder of this section, we review some relations between the homology of manifold and its $d$-fold branched or unbranched cover. The basic homological invariants for a $4$-manifold are the first and second homology, and the signature of the intersection form.  We will see that the combinatorics of a configuration places constraints on, and in some cases determines, these invariants for branched covers associated to a configuration.

We consider a cyclic group $G \cong \Z_d$ acting on a space $\xtilde$, with quotient map $\pi: \xtilde \to X = \xtilde/G$. Fixing a generator $\tau \in G$,  we decompose the homology groups $H_*(\xtilde;\C)$ into eigenspaces for the action of $\tau_*$. Writing $\omega = e^{2 \pi i/d}$, the $\omega^r$ eigenspace of the action of $\tau_*$ on $H_*(\xtilde;\C)$ is denoted $H_k(\xtilde;\C)^{(r)}$, and its rank by $b_k(r)$.  We will study mainly the case when $\xtilde$ is a $4$-manifold, and the action is semifree (free away from the fixed point set) with $2$-dimensional fixed point set. The fixed point set $\ftilde$ maps homeomorphically down to $F \subset X$, and we write $\vtilde \to V = X -F$ for the associated unbranched cover.  Define the $r$-Euler characteristic by
$$\chi_r(\xtilde)=\sum_k (-1)^k b_k(r)(\xtilde).$$
 
We will make use of the following facts~\cite{gilmer:configurations,edmonds:aspects}.
\begin{proposition}\label{P:unbranched}
With notations as in the preceding paragraph, we have:
\begin{enumerate}
\item $H_k(\vtilde;\C)^{(0)} \cong H_k(V;\C)$.
\item $\chi_r(\vtilde)=\chi(V)$\label{item2}.
\item If $d$ is a power of a prime $p$, then $b_1(r)(\vtilde) \leq \dim_{\Z_p}H_1(V; \Z_p) -1$ for $r \neq 0$.\label{item3}
\end{enumerate}
\end{proposition}
\begin{proof}
The first item is a standard transfer argument; the second and third are Propositions 1.1 and 1.5, respectively, in~\cite{gilmer:configurations}. Technically, Gilmer's Proposition 1.5 states \eqref{item2} and \eqref{item3} for $r=1$ but the result holds for any $r$: If $(r,d)= 1$, then it follows from~\cite[Lemma 7.2]{thomas-wood} which states that $b_1(r)(\vtilde) = b_1(1)(\vtilde)$. If $r| d$, then~\cite[Lemma 7.4]{thomas-wood} implies that $b_1(r)(\vtilde) = b_1(1)(\vtilde/(\Z_{d/r}))$.
\end{proof}
We can get similar results for the branched cover as well.

\begin{corollary}\label{C:branched}
Suppose that $X$ is a $4$-manifold, and that $F$ is a union of spheres each of which has non-trivial normal bundle.
\begin{enumerate}
\item $H_k(\xtilde;\C)^{(0)} \cong H_k(X;\C)$.
\item $H_k(\xtilde;\C)^{(r)} \cong H_k(\vtilde;\C)$ for $r \neq 0$.
\item If $d$ is a power of a prime $p$, then $b_1(r)(\xtilde) \leq \dim_{\Z_p}H_1(V; \Z_p) -1$.
\item The Euler characteristic of the $d$-fold branched cover $\xtilde$ is given by:
\begin{equation}\label{eqn:eulerchar}
\chi(\xtilde)=d(\chi(X)-2|F|)+2|F|.
\end{equation}
where $|F|$ denotes the number of components of $F$.
\end{enumerate}
\end{corollary}
\begin{proof}
Since $G$ acts as the identity on the homology of $N(\ftilde)$ and $\partial N(\ftilde)$, the first three items follow by a Mayer-Vietoris argument. The last item follows by a straightforward simplex-counting argument, or from the first two.
\end{proof}

When $X$ is a $4$-manifold, we can also consider the equivariant intersection form. Define $b_2^\pm(r)$ by restricting the intersection form on $H_2(\xtilde)\otimes \C$ to $H_2(\xtilde;\C)^{(r)}$, and let  
$$
\epsilon_r(\xtilde)= b_2^+(r) -b_2^-(r)
$$
denote the signature of the intersection form of $\xtilde$ restricted to the eigenspace. A standard transfer argument identifies $\epsilon_0(\xtilde)$ as the signature of the quotient $X = \xtilde/G$.  A calculation utilizing the G-signature Theorem~\cite{atiyah-singer:III,gordon:G-signature} yields the following result first computed by Rohlin~\cite{rohlin:genus} and proved in the format we will utilize by Casson and Gordon.

\begin{theorem}[Casson-Gordon \cite{cassongordon}]\label{thm:epsilonsignature}
	Suppose $\xtilde$ is a $d$-fold branched cover over $X$ whose branch set is a smooth (but possibly disconnected) submanifold $F\subset X$ with lift $\widetilde{F}\subset \xtilde$. If the generator of the covering action rotates the normal bundles to each component of $\widetilde{F}$ by $\pm 2\pi/d$ then
$$
\varepsilon_r(\xtilde)=\sign(X)-\frac{2r(d-r)}{d^2}[F]^2.
$$
\end{theorem}

\section{Line arrangement obstructions from branched coverings}
In this section we will find certain mod $d$ relations and calculate the invariants of their branched covers. The mod $d$ relations which have no linearly independent subrelations are most useful for our purposes because these correspond to branch covers with $b_1=0$ by Corollary \ref{C:branched}, so we can determine more about their topology. 

For the finite projective planes, $P(2,p)$, the mod $p$ relations of the proper transforms of the lines in any realization are specified by equations \ref{eqn:plines} and \ref{eqn:exceptional}. In terms of the incidence matrix $A_p$ for the combinatorial line arrangement $P(2,p)$, collection of equations \ref{eqn:exceptional} for $j=1,\cdots, N$ can be encoded as
$$\left[a_1 \cdots a_N \right] A_p = [0 \cdots 0] \mod p.$$
 The solutions form a well-studied linear code over $\FF_p$.  By interpreting the mod $p$ relations in the coding theory language, it is shown in \cite[Theorem 3]{bagchi-inamdar}  that the minimum weight among the nonzero relations for $P(2,p)$ is $2p$. In fact all of these weight $2p$ relations are examples in the more general family of sub-arrangements $\mathcal{B}_{\alpha,\beta}^p$ defined below. We use the line arrangement interpretation from section \ref{sec:linerels} to prove the more general statement that the relations in this family have no linearly independent sub-relations in Proposition \ref{prop:Babprels} below.

\subsection{A family of minimal mod $p$ relations}

Consider the family of line arrangements, $\mathcal{B}_{\alpha,\beta}^p$ with $p\alpha+p\beta$ lines arranged so that $L_1,\cdots L_{p\alpha}$ 
intersect at a single point $P_1$ and the other lines $L_{p\alpha+1},\cdots, L_{p\alpha+p\beta}$ all intersect at a different point $P_2$. The remaining points in the arrangement are double points where a line from the first group intersects a line from the second group. Note that there are a total of $2+p^2\alpha\beta$ points. See figure \ref{fig:palphabetarelation}. 

When $\alpha=\beta=1$, $\B_{1,1}^p$ is a strict sub-arrangement of $P(2,p)$. This can be seen by choosing any two distinct points in $P(2,p)$ to be $P_1$ and $P_2$, and then including all the lines through $P_1$ or $P_2$ except the unique line through both.

Now suppose that we blow-up at the $2+p^2\alpha\beta$ points in any $\C$-realization of $\B_{\alpha,\beta}^p$ and consider the proper transforms $L_1,\cdots, L_{p(\alpha+\beta)}$ of the lines. There is a mod $p$ relation with coefficient $1$ for $L_1,\cdots, L_{p\alpha}$ (the lines through $P_1$) and $-1$ for $L_{p\alpha+1},\cdots, L_{p\alpha+p\beta}$ (the lines through $P_2$). We can verify this is a mod $p$ relation by checking that at each intersection point, the sum of coefficients of the lines passing through $P_j$ is zero mod $p$. At $P_1$ and $P_2$, the sums are $\alpha p(1)$ and $\beta p(-1)$ which are clearly zero mod $p$. At any other vertex, $P_j$, there one line with with coefficient $1$ and another with coefficient $-1$ summing to $0$.

\begin{figure}
	\centering
	\includegraphics[scale=.4]{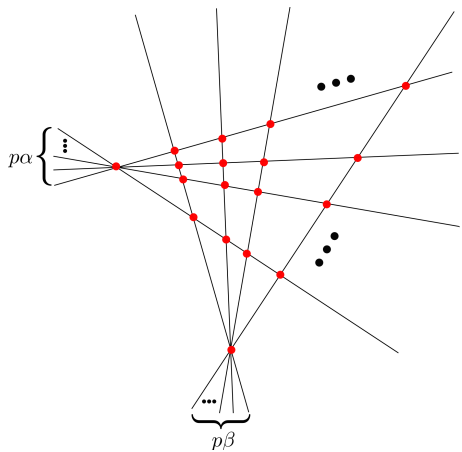}
	\caption{The $\B_{\alpha,\beta}^p$ sub-arrangement.}
	\label{fig:palphabetarelation}
\end{figure}

\begin{proposition}\label{prop:Babprels}
There are no mod $p$ relations which are linearly independent from $(1,\cdots, 1, -1,\cdots, -1)$ amongst the proper transforms of $L_1,\cdots , L_{p(\alpha+\beta)}$ (the lines of $\B_{\alpha,\beta}^p$) after blowing up at all intersection points of $\B_{\alpha,\beta}^p$.
\end{proposition} 
\begin{proof}
	Suppose there were coefficients $(a_1,\cdots, a_{p\alpha+p\beta})$ yielding a non-trivial mod $p$ relation. Then some $a_i\neq 0$. Without loss of generality, say $a_1\neq 0$. Then using equation \eqref{eqn:exceptional} for the double points along $L_1$ we find that $a_{p\alpha+1}\equiv \cdots \equiv a_{p\alpha+p\beta}\equiv -a_1 \mod p$. Using equation \eqref{eqn:exceptional} for the double points along $L_{p\alpha+1}$ we get that $a_2\equiv \cdots \equiv a_{p\alpha}\equiv a_1 \mod p$. Therefore all mod $p$ relations are multiples of $(1,\cdots, 1, -1, \cdots, -1)$.
\end{proof}

\subsection{The topology of the branched cover}
Here we calculate invariants of branch covers of blow-ups of $\cptwo$ with branch set the proper transform of the blow-up of a $\B_{\alpha,\beta}^p$ arrangement.

\begin{proposition}\label{prop:bettiBpab}
	Consider a $\C$-realization of a line arrangement $\A$ with a sub-arrangement $\B_{\alpha,\beta}^p$. Blow-up $\cptwo$ at $N$ points such that the points lying on $\B_{\alpha,\beta}^p$ are precisely the intersections between lines in that sub-arrangement (e.g. if $\B_{\alpha,\beta}$ is a strict sub-arrangement, the condition is that the blow-up points include all points of $\B_{\alpha,\beta}$). Let   $F$ denote the union of the proper transforms of the lines in $\B_{\alpha,\beta}^p$. Let $\xtilde$ be the $p$-fold branched cover of $X=\cptwo\conn N\cptwobar$ over $F$ associated to the $\mod p$ relation
	$$\sum_{j=1}^{p\alpha}F_j-\sum_{j={p\alpha+1}}^{p\alpha+p\beta}F_j.$$
	
	We have the following topological invariants for $\xtilde$;
	\begin{enumerate}
		\item $\chi(\xtilde)=p(3+N-2p(\alpha+\beta))+2p(\alpha+\beta)$.
		\item $b_1(\xtilde)=0$.
		\item $b_2^+(r)(\xtilde)=2+2\alpha\beta r(p-r)-p(\alpha+\beta)$ for $r\neq 0$.
		\item $b_2^-(r)(\xtilde)=1+N-2\alpha\beta r(p-r)-p(\alpha+\beta)$ for $r\neq 0$.
	\end{enumerate}
\end{proposition}

\begin{proof}
The first item is simply Equation \eqref{eqn:eulerchar} applied in this situation.

To prove item (2), note that by Proposition \ref{prop:Babprels} the only non-trivial mod $p$ relations of $\F$ are multiples of $(1,\cdots, 1, -1, \cdots, -1)$. Therefore by Corollary \ref{C:cyclic}, $H^1(X-F;\Z_p)$ has rank one. By items (1) and (3) of Corollary~\ref{C:branched}, $H_1(\xtilde;\C)^{(r)}$ vanishes for all $r$.  Poincar\'e duality implies the vanishing of $H_3(\xtilde;\C)^{(r)}$ as well. Equivalently, $b_1(r)(\xtilde)=b_3(r)(\xtilde)=0$ for all $r=0,\cdots, p-1$.

Since $\xtilde$ is a connected 4-manifold, $H_0(\xtilde)\cong H_4(\xtilde)\cong \Z$, and the $\Z/p$ action on these groups is trivial. So $b_0(0)(\xtilde)=b_4(0)(\xtilde)=1$ and $b_0(r)(\xtilde)=b_4(r)(\xtilde)=0$ for $r\neq 0$. It follows that $\chi_r(\xtilde)=b_2(r)(\xtilde)$ for $r\neq 0$.  On the other hand, item (2) of Corollary~\ref{C:branched} implies that $\chi_r(\xtilde)=\chi(X-F)$ for $r \neq 0$. Putting these together, we see that
$$
b_2(r)(\xtilde)=\chi(X-F)=3+N-2p(\alpha+\beta)
$$
 for $r\neq 0$.
 
Next we calculate $\varepsilon_r(\xtilde)$ using theorem \ref{thm:epsilonsignature}. For this, we need to know $[F]^2$. For each of the first $p\alpha$ lines in $X$ we have blown up at $p\beta+1$ points, so each of their proper transforms has self-intersection $-p\beta$. Similarly the next $p\beta$ proper transforms have self-intersection $-p\alpha$. In total $[F]^2=-2p^2\alpha\beta$. Therefore by theorem \ref{thm:epsilonsignature}
$$\varepsilon_r(\xtilde)=1-N+4\alpha\beta r(p-r).$$
Solving the equations $b_2(r)=b_2^+(r)+b_2^-(r)$ and $\varepsilon_r=b_2^+(r)-b_2^-(r)$ yields the resulting formula.
\end{proof}

We will use this information to obstruct realizations of certain line arrangements, by showing that the proper transforms of the other lines in the arrangement represent a subspace of the negative definite piece of $H_2(\xtilde)$ (or the $\omega^r$ eigenspace $H_2(\xtilde;\C)^{(r)}$) of rank larger than is allowed by the calculation of $b_2^-$ (or $b_2^-(r)$) above.

\section{Non-realizable line arrangements}

Here we demonstrate how to use the branch covering invariants to obstruct $\C$-realizations of line arrangements. We provide two types of examples which both use a $\B_{\alpha,\beta}^p$ sub-arrangement to produce the branching set, but differ in how the number of blow-ups is chosen. Specifically, with the finite projective planes we will blow-up at all intersection points in the configuration causing the proper transforms to be disjoint and thus obviously independent generators of second homology. In the second example, the $14_4$ configuration, we will only blow-up at the points at the intersections in the $\B_{\alpha,\beta}^p$ sub-arrangement. This gives us a branched covering with lower values of $b_2^-(r)$ but we need to check by hand that the intersection form restricted to the subspace spanned by the spheres outside the branch set has full rank. Note that the differences in these two examples are essential for obtaining obstructions to their realizations. In any given example, one should consider the various options of whether and how much to blow-up away from the branch set to optimize the chances of finding an obstruction to realization.

\subsection{Finite projective planes}
\begin{proof}[Proof of Theorem \ref{T:projplane}]
Recall that by the comment at the beginning of section \ref{sec:examples}, to obstruct realizability of the $P(2,p^q)$ combinatorial line arrangements it suffices to obstruct realizability of $P(2,p)$.	

Assume that we have a topological $\C$-realization of the combinatorial line arrangement $P(2,p)$ in $\cptwo$. Let $X=\cptwo\conn(p^2+p+1)\cptwobar$ where the blow-ups are performed precisely at the intersections of the lines in the arrangement. As noted above, $\B_{1,1}^p$ is a strict sub-arrangement of $P(2,p)$. Therefore we can construct a $p$-fold branched cover $\xtilde$ as in Proposition \ref{prop:bettiBpab} with $N=p^2+p+1$. 

The Euler characteristic of $\xtilde$ is
$$\chi(\xtilde) = p(p^2+p+4-2(2p))+2(2p)=p(p^2-3p+8)$$
and
$$b_2(\xtilde)=p(p^2-3p+8)-2.$$

On the other hand, there are $p^2-p+1$ lines in $P(2,p)$ which are not part of $\B_{1,1}^p$ and these become disjoint from the branch set and from each other in $X$. Each line contains $p+1$ points which are blown-up so the proper transforms have square $-p$. Thus there are $p(p^2-p+1)$ disjoint $(-p)$-spheres and $2p$ disjoint $(-1)$-spheres in the cover $\xtilde$. Since these are disjoint spheres of non-zero self-intersection, the intersection form restricted to their span is non-degenerate of rank $p(p^2-p+1)+2p=p(p^2-p+3)$. Therefore $b_2(\xtilde)\geq p(p^2-p+3)$.

We find a contradiction here when $p>2$ because $$b_2(\xtilde)=p(p^2-3p+8)-2<p(p^2-p+3).$$

When $p=2$, $b_2(\xtilde)=10$, and our assumptions imply $\xtilde$ contains $6$ disjoint $(-2)$-spheres and $4$ disjoint $(-1)$-spheres for a total of $10$ independent generators of $H_2(\xtilde)$. Here we can rule out this possibility using the signature. In particular by Proposition \ref{prop:bettiBpab} $b_2^-(1)=2$ and $b_2^-(0)=b_2^-(X)=7$.

Therefore $b_2^- = 9$, so we cannot have $10$ linearly independent generators of the negative definite subspace of $H_2(\xtilde)$. Hence we find a contradiction to the existence of a topological $P(2,p)$ arrangement in $\cptwo$.

\end{proof}

In fact we can strengthen this theorem further using the full power of proposition \ref{prop:bettiBpab}.
\begin{theorem}
	Let $\A$ be a combinatorial line arrangement obtained from $P(2,p)$ for $p>2$ by deleting fewer than $(p^2-3)/2$ lines which are disjoint from a fixed $\B_{1,1}^p$ sub-arrangement. Then $\A$ is not topologically $\C$-realizable.
\end{theorem}

\begin{proof}
	Suppose we had a realization of $\A$ in $\cptwo$. Blow-up at $p^2+p+1$ points including each intersection point of $\A$ such that exactly $p+1$ points are blown up on each line of $\A$. This is possible because $\A$ is a sub-arrangement of $P(2,p)$. Take a branch cover over the proper transforms of the lines in the fixed $\B_{1,1}^p$ sub-arrangement.
	
	Then there are ($|\A_\Lc|-2p$) other $(-p)$-spheres which are not part of the branch set. Each of these will have $p$ lifts $S_0,\cdots, S_{p-1}$ which are cyclically permuted by the $\Z_p$ action. Therefore, for each of these spheres, the combination $S_0+\omega^{(p-1)r}S_1+\cdots+\omega^{2r}S_{p-2}+\omega^r S_{p-1}$ is in the $\omega^r$-eigenspace of the generator of the $G$-action ($r\neq0$). Such an element has square
	\begin{multline*}
	[S_0+\omega^{(p-1)r}S_1+\cdots+\omega^{2r}S_{p-2}+\omega^r S_{p-1}] \Bigcdot\\
	\overline{[S_0+\omega^{(p-1)r}S_1+\cdots+\omega^{2r}S_{p-2}+\omega^r S_{p-1}]}
	=-p^2.
\end{multline*}
		In particular these elements have negative square. Since the ($|\A_\Lc|-2p$) spheres downstairs are all disjoint from each other, these lifts are all independent of each other in $H_2(\xtilde;\C)^{(r)}$. Therefore $b_2^-(r)\geq |\A_\Lc|-2p$ for all $r\neq 0$.
	
	On the other hand proposition \ref{prop:bettiBpab} tells us $b_2^-(r)(\xtilde)=1+(p^2+p+1)-2r(p-r)-2p$ for $r\neq 0$. This is minimized when $r=(p-1)/2$:
	$$b_2^-((p-1)/2)=(p^2+p+1)-2p-\frac{p^2-3}{2}.$$
	
	Thus if we have deleted less than $(p^2-3)/2$ lines from $P(2,p)$ to obtain $\A$, $|\A_\Lc|-2p>(p^2+p+1)-2p-(p^2-3)/2=b_2^-(r)(\xtilde)$ and we obtain a contradiction.
\end{proof}

\subsection{The $14_4$ configuration}

As mentioned in the introduction, there is a unique combinatorial $14_4$ configuration of $14$ lines. There are $14$ distinguished points of multiplicity $4$ so that each line contains exactly $4$ of these. A given line intersects each of the other $13$ lines once, and the four quadruple points on a given line account for the intersection of that line with $12$ other lines. Therefore there is exactly one additional intersection point on each line and it is a double point. We show here that such an arrangement is not topologically $\C$-realizable.

\begin{proof}[Proof of Theorem \ref{T:144config}]
	First, notice that $\B_{2,2}^2$ is a $2$-relation that is a sub-arrangement of the $(14_4)$ arrangement by the following argument. Choose one of the quadruple points, and include the four lines which pass through it. To find the $\B_{2,2}^2$ sub-arrangement, it suffices to find another quadruple point in the configuration which does not lie on any of these four lines. Each line contains four quadruple points, so there are $13$ quadruple points in total which lie on these four lines. Thus there is exactly one quadruple point which is not on any of these lines so $\B_{2,2}^2$ is a sub-arrangement. Moreover, every line in the $(14_4)$ arrangement contains five intersection points, and every line in the $\B_{2,2}^2$ arrangement contains five intersection points (four double points and one quadruple point). Therefore $\B_{2,2}^2$ is a strict sub-arrangement of the $(14_4)$ arrangement.
	
	There are six lines and three points in $14_4$ which are not in $\B_{2,2}^2$. Suppose we could realize the $14_4$ arrangement topologically in $\cptwo$. Then blow-up at only the $18$ points in the $\B_{2,2}^2$ sub-arrangement. Let $S_1,\cdots, S_6$ be the proper transforms of the lines which are not in the $\B_{2,2}^2$ sub-arrangement. Because $\B_{2,2}^2$ was a strict sub-arrangement, these spheres are disjoint from the proper transforms of the $\B_{2,2}^2$ spheres which will be our branch set. Then intersection form restricted to their span in that basis is
	$$\oplus_3 \left[ \begin{array}{cc} -3 & 1 \\ 1 & -3 \end{array} \right]$$
	which is non-degenerate. Take the $2$-fold cover $\xtilde$ as above. Then each of the spheres $S_i$ has two lifts $S_i^0$ and $S_i^1$ which are interchanged under the deck transformation. The classes $[S_i^0]-[S_i^1]$ are in the $-1$-eigenspace. The intersection form on $([S_1^0]-[S_1^1],\cdots, [S_6^0]-[S_6^1])$ is
	$$\oplus_3 \left[ \begin{array}{cc} -6 & 2 \\ 2 & -6 \end{array} \right]$$
	which is still non-degenerate and thus has rank 6. Therefore $b_2^-(1)\geq 6$. On the other hand, applying proposition \ref{prop:bettiBpab} shows that $b_2^-(1)=1+18-8-8=3$ which is a contradiction.
	
\end{proof}

\section{Complexifying real pseudolines}
A geometric real line arrangement is specified by a collection of real linear equations; the solutions to these in $\cptwo$ form a complex line arrangement with the same combinatorics.  The passage from the real to complex arrangement is called {\em complexification}.  While an \rpl\ arrangement is specified by a collection of non-linear equations, passing to the complex solutions of these is unlikely to yield a \cpl\ arrangement, let alone one with the same combinatorics.  So it is a reasonable question to ask if there is a more topological complexification process; in this section we prove Theorem~\ref{T:real-complex} and show that indeed there is. Additionally, we prove Theorem~\ref{T:real-symplectic}, showing that after isotoping the \rpl\ arrangement, our complexification process will yield a symplectic line arrangement.

\subsection{The smooth construction}
The proof of Theorem~\ref{T:real-complex} is based on a well-known decomposition of $\cptwo$ as the union of tubular neighborhoods of the standard copy of $\rptwo$ and a neighborhood of a smooth conic $C$ having no real points.  For concreteness, we take $C$ to be the Fermat curve given by $z_0^2 + z_1^2 + z_2^2 =0$.  Since $C$ is defined over the reals, this decomposition may be taken to be invariant under complex conjugation. This decomposition is described nicely in~\cite{gilmer:real}, and we will make use of several formulas from that paper. 

Observe that there is an orientation reversing diffeomorphism between $T\rptwo$ and the normal bundle $\nu_{\rptwo}$ in $\cptwo$. This diffeomorphism is given by multiplication by $i$. Equivalently, there is an orientation preserving diffeomorphism between $T^*\rptwo$ and $\nu_{\rptwo}$. This also follows from the Weinstein neighborhood theorem because $\rptwo$ is Lagrangian in $\cptwo$ with the standard symplectic structure.

Gilmer~\cite[Remark 4.2]{gilmer:real} gives an explicit map $\Psi$ from $\bar{E}$, the unit tangent disk bundle of $\rptwo$, to $\cptwo$. 
He shows that is a $\Psi$ is diffeomorphism from the interior of the bundle to $\cptwo - C$, and $\Psi: \partial \bar{E} \to C$ is an oriented circle bundle, readily identified with the unit normal bundle of $C$ in $\cptwo$. From this picture, we can see the complexification of an actual line $\ell$ in $\rptwo$. Parameterize $\ell$ as the quotient of a unit-speed curve $\gamma(t) \subset S^2 \subset \R^3$. Then the complexification $\ell_\C$ is parameterized as
\begin{equation}\label{E:LC}
\left\{\Psi(\gamma(t),s \gamma'(t)),\ s \in [-1,1],\ \text{and}\ t\in [0,2\pi]\right\}.
\end{equation}
Note that the points with $s = \pm 1$ represent two distinct intersections of $\ell_\C$ with $C$ that are interchanged by complex conjugation.

We may generalize this to non-straight curves in $\rptwo\subset \cptwo$. Gilmer describes how any such curve (or collection of curves) gives rise to a link in the boundary of a neighborhood of $\rptwo$. Let $\widetilde{Q}$ be the unit circle bundle in $\nu_{\rptwo}$ (identified with the unit co-tangent bundle). The {\em lift} of a curve $\Gamma$ in $\rptwo$ is the link $L(\Gamma)\subset \widetilde{Q}$ given by
$$L(\Gamma):=\{(p,v)\in \widetilde{Q}\mid p\in \Gamma, v^* \in T_p\Gamma, ||v||=1\}.$$
Note that there are two points in each unit cotangent fiber over each $p\in \Gamma$.

We make the following basic but important observation which follows immediately from the definition.

\begin{lemma}
	Suppose $\Gamma_1$ and $\Gamma_2$ are curves in $\rptwo$ intersecting transversally. Then $L(\Gamma_1)$ and $L(\Gamma_2)$ are disjointly embedded links in $\widetilde{Q}$.
\end{lemma}

We will say that a family of immersed curves in $\rptwo$, $\{\Gamma_j^t\}$ for $t\in [0,1]$ and $j\in \{1,\cdots, n\}$ is a \emph{transverse regular homotopy} if $\Gamma_i^r$ and $\Gamma_j^r$ have only transverse intersections for each $r\in[0,1]$ and $i\neq j$.  (In all of the regular homotopies considered in this paper, the individual curves will remain embedded.)  A useful version of the previous lemma in this context is the following.

\begin{lemma}
	Let $\{\Gamma^t_j\}^{t\in [0,1]}_{j\in \{1,\cdots, n\}}$ be a transverse regular homotopy of curves in $\rptwo$. Then $L(\Gamma_1^t)\cup\cdots \cup L(\Gamma_n^t)$ is an embedded link for all $t\in [0,1]$. In other words, a transverse regular homotopy of curves in $\rptwo$ lifts to an embedded isotopy of links in $\widetilde{Q}$.
\end{lemma}

Tracing out this isotopy as we vary the radius of the circle fibers in the cotangent bundle, we can build a surface analogous to the complexification as in equation \eqref{E:LC}. For each  $r\in (0,1)$, the circle bundle of radius $r$ is diffeomorphic to $\widetilde{Q}$ and there is a corresponding link $L_r(\Gamma^r_j)$ lifting each curve at the time in the isotopy corresponding to the radial coordinate. (We will assume the isotopy is the identity in a neighborhood of the end-points $0$ and $1$.) The union of these links over all $r$ together with the core curve $C$ form an annulus in the disk bundle:
$$A(\Gamma^t):=\{(p,v)| p\in \Gamma^r, v^*\in T_p\Gamma^r, ||v||=r, r\in [0,1) \}.$$

To cap off these annuli to spheres in the class of the complex projective line, we would want the link
$$L(\Gamma^1)=\{(p,v)| p\in \Gamma^1, v^*\in T_p\Gamma^1, ||v||=1 \}$$
to be a collection of circle fibers over the conic. However, if the curves $\Gamma^r$ are straight lines in $\rptwo$ for $r>1-\varepsilon$ for some $\varepsilon>0$, then $A(\Gamma^t)$ agrees with the complexification of a straight line (as in equation \eqref{E:LC}) sufficiently close to the conic. Therefore, we can smoothly cap off with disks attached to the two ends of $A(\Gamma^t)$. 

Denote the resulting spheres by $S(\Gamma^t)$. We note that in this case, the disks are the normal fibers to the conic $C$ which are complex and thus symplectic disks.

\begin{lemma}
	Suppose $\Gamma_1^t$ and $\Gamma_2^t$ are homotopies of curves in $\rptwo$ such that $\Gamma_1^r$ and $\Gamma_2^r$ intersect transversally for all $r\in [0,1]$. Then $A(\Gamma_1^t)$ and $A(\Gamma_2^t)$ intersect transversally precisely at the intersection points of $\Gamma_1^0$ and $\Gamma_2^0$ in $\rptwo$.
\end{lemma}

\begin{proof}
	Suppose $(p,v)\in A(\Gamma_1^t)\cap A(\Gamma_2^t)$ and $||v||=r$. Then because $\Gamma_1^r$ and $\Gamma_2^r$ intersect transversally, $L(\Gamma_1^r)$ and $L(\Gamma_2^r)$ are disjoint in each circle bundle of fixed radius $r>0$, therefore $v=0$ and $p\in \Gamma_1^0\cap \Gamma_2^0$. Now at a point $p\in \Gamma_1^0\cap \Gamma_2^0$, the tangent space to $A(\Gamma_1^t)$ at $(p,0)$ contains the unit tangent vector $v\in T_p\Gamma_1^0$ and $iv \in \nu_p(\rptwo)$. Similarly, the tangent space to $A(\Gamma_2^t)$ at $(p,0)$ contains the unit tangent vector $w\in T_p\Gamma_2^0$ and $iw \in \nu_p(\rptwo)$. Since $\Gamma_1^0$ and $\Gamma_2^0$ intersect transversally, $v$ and $w$ are linearly independent in $T_p(\rptwo)$ therefore $(v,iv,w,iw)$ span $T_{(p,0)}\nu_{\rptwo}$ so $A(\Gamma_1^t)$ and $A(\Gamma_2^t)$ are transverse.
\end{proof}

To gain more explicit control over the pseudolines in $\rptwo$, we will utilize a nice positioning of an \rpl configuration, due to J.~Goodman~\cite{goodman:wiring}. 

A \emph{wiring diagram} with $n$ wires is a configuration of $n$ smooth strands in $[a,b]\times[c,d]$ which are graphical (the intersection of each of the $n$ strands with ${t}\times[c,d]$ is a single point), such that if the strands are labeled $1$ to $n$ from top to bottom along $\{a\}\times[c,d]$ then their corresponding end-points on $\{b\}\times[c,d]$ come in the reverse order: $n$ to $1$ from top to bottom. See figure \ref{fig:wiring}.

Note that while some of the drawings and constructions will be piecewise linear, we will always assume that we have done an arbirarily small rounding of the corners to smooth the curves. The rounding will be done in a neighborhood of the corner disjoint from the other lines, in particular we will always ensure the corners are disjoint from the other pseudolines.

We can identify $\rptwo$ with a compactification of $\R^2$ where a curve becomes a closed loop if its two ends approach infinity at the same slope. Let $\ell$ be the straight line connecting the two endpoints of the wire segment. Extend the wire by $\ell$ outside of the finite box shown in figure \ref{fig:wiring} as in figure \ref{fig:extendedwires}. Because of the reversal of the ordering, a wiring diagram glues up to give  an \rpl configuration, and Goodman~\cite{goodman:wiring} proves the converse: any real pseudoline arrangement is isotopic to a wiring diagram.
\begin{figure}[htb]
	\labellist
	\small\hair 2pt
	\pinlabel {$1$} [ ] at 105 528
	\pinlabel {$2$} [ ] at 105 512
	\pinlabel {$3$} [ ] at 105 496
	\pinlabel {$4$} [ ] at 105 480
	\pinlabel {$5$} [ ] at 105 464
	\pinlabel {$1$} [ ] at 439 464
	\pinlabel {$2$} [ ] at 439 481
	\pinlabel {$3$} [ ] at 439 497
	\pinlabel {$4$} [ ] at 439 514
	\pinlabel {$5$} [ ] at 439 528
	\endlabellist
	\centering
	\includegraphics[scale=.8]{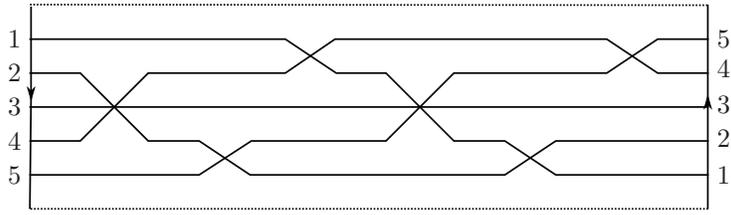}
	\caption{Wiring diagram}
	\label{fig:wiring}
\end{figure}

\begin{figure}
	\centering
	\includegraphics[scale=.6]{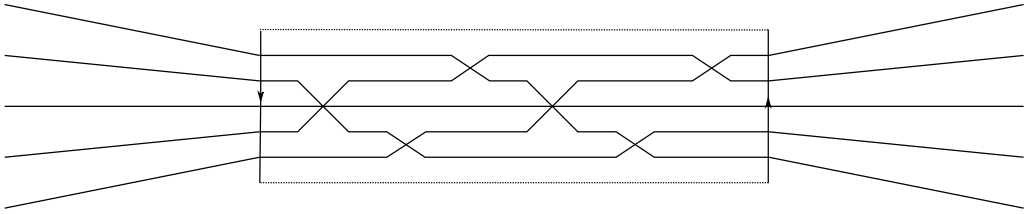}
	\caption{Extension of a wiring diagram to be viewed in an $\R^2$ chart in $\rptwo$ instead of on the Mobius band. The two ends of each straight line extension meet at infinity.}
	\label{fig:extendedwires}
\end{figure}

Now we prove the key remaining piece needed to extend an \rpl\ arrangement $\Gamma_1,\cdots, \Gamma_n$ to a smooth \cpl\ arrangement $S(\Gamma_1^r),\cdots, S(\Gamma_n^r)$. Theorem \ref{T:real-complex} follows.

\begin{proposition}\label{prop:transverse}
	Let $\Gamma_1,\cdots, \Gamma_n$ be a collection of real pseudolines. Then there is a transverse regular homotopy taking $\Gamma_1,\cdots, \Gamma_n$ to a straight line arrangement in $\rptwo$.
\end{proposition}

\begin{proof}
	Observe that if a homotopy realizes a tangency that cannot be avoided by an arbitrarily small smooth perturbation, then for some period of time during the homotopy the two pseudolines will intersect in at least three points transversally (the extra intersections are algebraically cancelling pairs). In particular, any planar isotopy of an \rpl\ arrangement (viewing the configuration of pseudolines as a planar graph where intersections are vertices) certainly can be realized as a transverse regular homotopy. Thus, we can get to a wiring diagram format by a transverse regular homotopy by Goodman's result.
	
	Next, we show that there is a transverse regular homotopy between a straight line arrangement of $k$ lines intersecting at a point and $k$ lines which intersect generically in double points such that the double points are ordered in a wiring diagram from left to right as follows. If the lines are labeled from top to bottom from $1$ to $k$ and we label the intersection between line $i$ and line $j$ by $(i,j)$ for $i<j$ then the intersections should have the dictionary order going from left to right in the wiring diagram (see figure \ref{fig:doublepoints}). We will say such a double point configuration of $k$ lines is canonically ordered.
	
	\begin{figure}
		\centering
		\includegraphics[scale=.8]{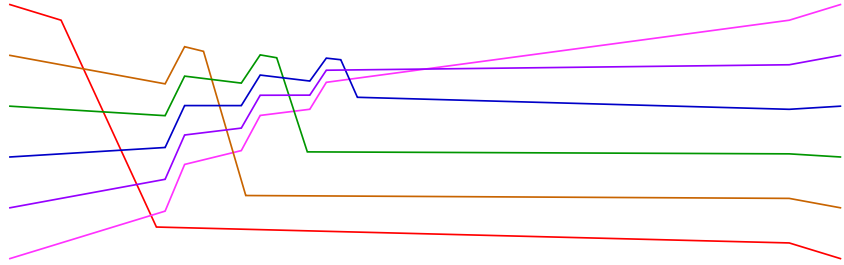}
		\caption{Canonically ordered double point configuration.}
		\label{fig:doublepoints}
	\end{figure}
	
	We homotope the lines one at a time, starting with the line with the highest endpoint on the left hand side. Choose a point on this line to the right of the multipoint to move so that the slope of the line to the left of this point becomes more steeply negative, and to the right of this line goes to zero except in a neighborhood of the right hand end point, as in figure \ref{fig:doublepointshomotopy}. Since all intersections with other lines only occur in the region where the line becomes more steeply negative than the slopes of any other line, this can be done transversally. By a planar isotopy, readjust the heights of the remaining $k-1$ lines to the right of their intersections with the first line to resymmetrize the heights in a wiring diagram neighborhood of the $(k-1)$-fold multipoint. In this smaller neighborhood we now have a standard intersection of multiplicity $k-1$. Thus by induction, we can achieve the desired transverse regular homotopy. Note that while this description uses piecewise linear curves, we can perform arbitrarily small roundings of corners, which we have ensured never intersect the other lines, to smooth the pseudolines without affecting the transversality property.

	\begin{figure}
		\centering
		\includegraphics[scale=.6]{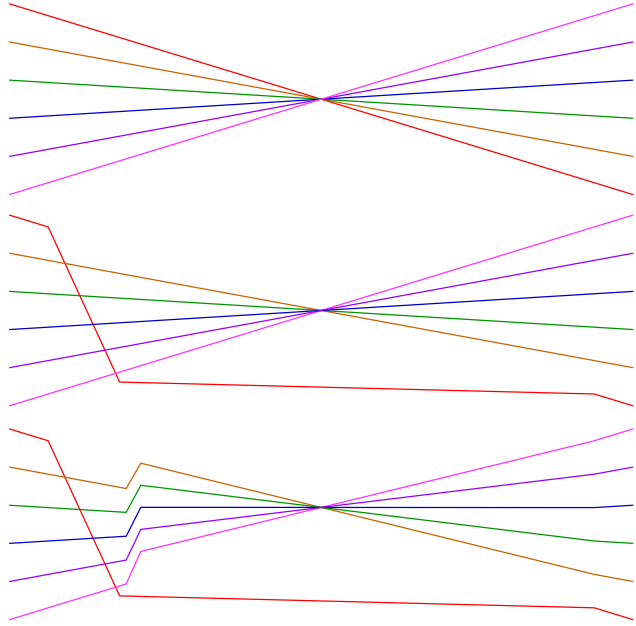}
		\caption{Inductive step for the regular transverse homotopy.}
		\label{fig:doublepointshomotopy}
	\end{figure}
	
	Now we can use this regular homotopy locally near every multi-intersection point of our original \rpl\ arrangement. This gives us a regular homotopy to an \rpl\ arrangement in wiring diagram form with only double point intersections. The double points of the full arrangement of $n$ lines may not come in the canonical order. However, by a classical theorem of Matsumoto and Tits (see \cite{bump} Theorem 25.2 or \cite{matsumoto}), we can reorder the double points to the canonical order by performing a sequence of \emph{braid moves}. To apply this theorem we reinterpret a wiring diagram with only double points as a reduced word representing the element of the symmetric group $S_n$ which takes $12\cdots n$ to $n\cdots 21$. The letters of the word are transpositions corresponding to the double point crossings. The reduced condition corresponds to the fact that any pair of wires crosses at most once because we are working with \rpl\ arrangements. Let $\tau_i$ be the transposition of the element in the $i^{th}$ and $(i+1)^{st}$ positions. Then the two braid relations are
	\begin{enumerate}
		\item $\tau_i\tau_j\tau_i=\tau_j\tau_i\tau_j$ when $|i-j|=1$; 
		\item $\tau_i\tau_j=\tau_j\tau_i$ when $|i-j|> 1$.
	\end{enumerate}
	The first move geometrically corresponds to a Reidemeister three move if interpreting the diagram as a knot projection. We can see this move in our context in figure \ref{fig:braid1}. This move can be achieved (after a planar isotopy) by collapsing the three double points to a single triple point by reversing in time the above construction. By rotating the same homotopy by $180^\circ$ (now going forward in time) we split the triple point into double points in a way that is planar isotopic to the right hand side of braid move (1) (see figure \ref{fig:braid1homotopy}). The second braid move corresponds simply to a planar isotopy which changes the relative horizontal position of crossings between disjoint pairs of strands (figure \ref{fig:braid2}).
	
	\begin{figure}
		\centering
		\includegraphics[scale=.8]{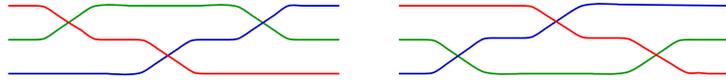}
		\caption{Braid move (1).}
		\label{fig:braid1}
	\end{figure}
	
	\begin{figure}
		\centering
		\includegraphics[scale=.8]{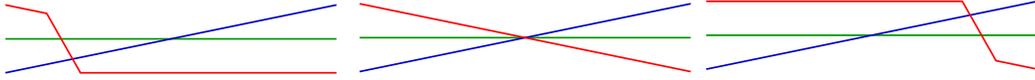}
		\caption{Transverse regular homotopy achieving braid move (1).}
		\label{fig:braid1homotopy}
	\end{figure}
	
	\begin{figure}
		\centering
		\includegraphics[scale=.8]{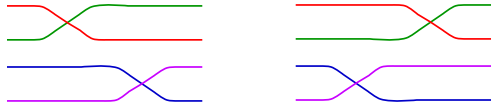}
		\caption{Braid move (2).}
		\label{fig:braid2}
	\end{figure}
	
	Therefore by Matsumoto's theorem, we have a transverse regular homotopy of our pseudoline arrangement to a wiring diagram arrangement with only double points occurring in the canonical order. Thus by reversing in time the transverse regular homotopy constructed above, we can transversally homotope this configuration to a straight line arrangement of $n$ lines intersecting at a single point of multiplicity $n$.  If we start with a wiring diagram, this transverse regular homotopy is compactly supported inside the wiring diagram box, and throughout the homotopy we have a wiring diagram form and the slopes of the wires near the boundary of the box agree with the extensions of these wires as in figure \ref{fig:extendedwires}. Thus the final configuration corresponds to a geometric straight line arrangement.	
\end{proof}

\subsection{Symplectifying}
	Next we determine when $S(\Gamma_i)$ is a symplectic submanifold of $\cptwo$ and prove Theorem \ref{T:real-symplectic}. 
		By a planar isotopy, we will assume $\Gamma_1,\cdots, \Gamma_n$ are arranged via a wiring diagram. Using the construction of proposition \ref{prop:transverse}, we will assume the transverse regular homotopy is supported in a compact disk, and has wiring diagram form throughout the homotopy. Then for each $i$, $S(\Gamma_i)$ coincides with a complex projective line outside of a single coordinate chart $\C^2\subset \cptwo$. Thus it suffices to check whether $S(\Gamma_i)$ is symplectic in this coordinate chart. Note that the intersection of this chart with $\rptwo$ is $\R^2$.
	
	We will use coordinates $(q_1,q_2)$ on $\R^2$ and $(p_1,p_2)$ on $i\R^2$ where the symplectic form is $\omega=dq_1\wedge dp_1+dq_2\wedge dp_2$. Assume that $q_1$ is the horizontal axis of the wiring diagram so that the pseudoline $\Gamma^r$ can be parametrized in the $(q_1,q_2)$ plane by $(t,q(r,t))$  where $q(r,t):=q_{|r|}(t)$ for $r\in[-1,1]$. Then we can parametrize $\A(\Gamma)$ by
	$$\psi(r,t)=\left(t,q(r,t), \frac{r}{\sqrt{1+\left(\frac{dq}{dt}\right)^2}},\frac{r\frac{dq}{dt}}{\sqrt{1+\left(\frac{dq}{dt}\right)^2}} \right).$$
	To determine whether $S(\Gamma)$ is symplectic we want to know that $\omega(\frac{d\psi}{dt}, \frac{d\psi}{dr})>0$ for all $(r,t)$. We compute this value to be:
	$$\left(1+\left(\frac{dq}{dt}\right)^2\right)^{-3/2}\left[ \left(1+\left(\frac{dq}{dt}\right)^2\right)^2+r\left(-\frac{dq}{dr}\frac{d^2q}{dt^2}+\frac{d^2q}{dtdr}\left(\frac{dq}{dt}\right)^2\left(\frac{dq}{dt}-1\right) \right) \right].$$
	
	Since $(1+\left(\frac{dq}{dt}\right)^2)>1$ it suffices to have 
	$$r\left(-\frac{dq}{dr}\frac{d^2q}{dt^2}+\frac{d^2q}{dtdr}\left(\frac{dq}{dt}\right)^2\left(\frac{dq}{dt}-1\right) \right)$$
	close to $0$. We can assume we started with an isotoped version of the wiring diagram which was stretched arbitrarily long by $1/\varepsilon$ in the $t$ direction. This has the effect of a change of variables replacing $t$ by $\varepsilon t$. Performing this change of variables gives 
\begin{multline*}
\omega(\frac{d\psi}{dt},\frac{d\psi}{dr})=
\left(1+\left(\varepsilon\frac{dq}{dt}\right)^2\right)^{-3/2} \times \\
\left[ \left(1+\left(\varepsilon\frac{dq}{dt}\right)^2\right)^2+r\left(-\varepsilon^2\frac{dq}{dr}\frac{d^2q}{dt^2}+\varepsilon\frac{d^2q}{dtdr}\left(\varepsilon\frac{dq}{dt}\right)^2\left(\varepsilon\frac{dq}{dt}-1\right) \right) \right].
\end{multline*}
	By making $\varepsilon$ sufficiently small, this value will be positive. Determine a sufficiently small value of $\varepsilon$. Replace the entire homotopy including the starting configuration with a new stretched version and adjust the slopes of the extended lines at infinity accordingly. The smooth complexification process agrees with the algebraic complexification wherever the lines are straight so the symplectic condition holds outside of a compact subset of this chart. Thus the $1/\varepsilon$ stretched version of $S(\Gamma^r)$ is symplectic everywhere. For a collection $\Gamma_1^r,\cdots, \Gamma_n^r$ determine an appropriate $\varepsilon_i$ for each and use $\min_i{\varepsilon_i}$ to stretch the entire configuration. This concludes the proof of theorem \ref{T:real-symplectic}.

\begin{remark}
The spheres in the configuration constructed in Theorem \ref{T:real-symplectic} are {\em flexible curves} of degree $1$ in the terminology of \cite{viro:real}. In addition to the conjugation-invariance proved in that theorem, flexibility entails that the restriction to each $\Gamma_i$ of the tangent field to each surface be isotopic (through conjugation-invariant plane fields) to the restriction of the tangent field of a complex line. It is readily verified that this condition holds in our construction.
\end{remark}

\section{Non-fillable contact manifolds}
\label{S:non-fillable}

Symplectic fillings are symplectic manifolds with positive contact type boundary, also known as convex boundary. Contact type is useful because then a contact structure determines the germ of the symplectic structure along the boundary using considerably less data than the symplectic form itself. Convex positive contact boundaries can be glued to \emph{concave} negative contact boundaries to get a global symplectic structure on the glued manifold. Symplectic fillings tend to have more topological restrictions than concave caps, their negative contact boundary counterparts. Understanding which smooth manifolds can be symplectic fillings of a given contact boundary is both an interesting classification problem in its own right, and is useful for producing symplectic cut and paste operations to construct closed symplectic manifolds with interesting properties (see \cite{park, karakurtstarkston} for a small sample of examples). Most of the classifications of symplectic fillings which have been achieved, eventually come down to understanding a classification of symplectic curves in $\cptwo$ with prescribed singularities (see \cite{mcduff:rationalruled, lisca, ohtaono:symplecticfillings, ohtaono:simplesingularities, starkston:fillings, gollalisca:stein}.  Symplectic \cpl\ arrangements are an example of such curves (the totally reducible case) which appear the second author's classification of fillings of a large class of Seifert fibered spaces over $S^2$ \cite{starkston:fillings}. Here we give a different symplectic filling result which is more immediately related to symplectic realizability of a line arrangement.

If there is a symplectic realization of a line arrangement in $\cptwo$, then it will have a concave symplectic neighborhood, and the complement of that neighborhood will be a strong symplectic filling of the dividing contact hypersurface. Even without an actual symplectic realization of the line arrangement in $\cptwo$, we can construct a manifold $N_\A$ which deformation retracts onto a configuration of intersecting spheres, such that those spheres have intersections and self-intersections specified by a given combinatorial line arrangement $\A$. We show that this manifold with boundary can be endowed with a concave symplectic structure. We will then study this concavely induced contact manifold $(Y_\A,\xi_\A)$.

\begin{proposition}
	Let $\A$ be a combinatorial line arrangement. Then there exists a symplectic manifold $N_\A$ with concave boundary which deformation retracts onto a collection of transversally intersecting spheres, such that the spheres intersect according to $\A$ and each sphere has self-intersection number $1$.
\end{proposition}

\begin{proof}
	A plumbing of disk bundles over surfaces is a way to build a manifold with boundary which deformation retracts onto a collection of surfaces which intersect transversally in double points. While we could build $N_\A$ in a similar way using a local model for each intersection multiplicity, we instead use blow-ups to directly appeal to plumbing constructions. This allows us to quickly adapt symplectic results about plumbings to our situation.
	
	We construct a plumbing $P_\A$ which will contain a collection of $-1$ (exceptional) spheres in correspondence with the intersection points of $\A$ of multiplicity at least $3$. The key property of $P_\A$ is that the image of the core spheres of the plumbing, after blowing down this collection of exceptional spheres, will be $+1$ spheres intersecting each other as specified by $\A$. The existence of $P_\A$ follows from the elementary fact that if $n$ surfaces intersect transversally at a point, then after blowing up, the $n$ surfaces become disjoint from each other, and they each intersect the exceptional sphere transversally at different points. Thus we build $P_\A$ by trading each multiplicity $n$ point in $\A$ for an exceptional sphere which intersects each of the $n$ spheres in distinct double points. We decrease the self-intersection number from $+1$ by one for each exceptional sphere it intersects.
	
	Let $k$ be the number of lines in $\A$ and let $N$ be the number of intersection points of multiplicity at least $3$ in $\A$. Then $P_\A$ is a plumbing of $k+N$ disk bundles over spheres ($k$ for the proper transforms of the lines and $N$ for the exceptional spheres).

	Explicit constructions of symplectic structures on plumbings were first studied by Gay and Stipsicz in \cite{gaystipsicz:convexnbhds} to produce symplectic negative definite plumbings with convex boundary. This construction was extended by Li and Mak to the concave case in \cite{limak:concavenbhd}, where it is shown that a plumbing supports a symplectic structure with concave boundary if the plumbing graph satisfies the \emph{positive G-S criterion}. This means that if $Q$ is the $(k+N)\times (k+N)$ intersection matrix for the plumbing, there exists $z\in (0,\infty)^{k+N}$ such that $Qz$ is a vector with all positive components.
	
	Ordering the vertices so that the first $k$ correspond to proper transforms of lines $L_1,\cdots, L_k$ and the last $N$ correspond to exceptional spheres, we see that $Q$ takes the following block form:
	$$Q=\left[  \begin{array}{cc} B & A \\ A^T & -I_N \end{array} \right].$$
	Here $A$ is the $k\times N$ incidence matrix for the $k$ lines and $N$ multi-intersection points of $\A$ (the double point incidence columns are left off), and $B$ is the following $k\times k$ symmetric matrix.
	$$B= \left[ \begin{array}{cccc} 
	1-n_1 & 1-\delta_{1,2} & \cdots & 1-\delta_{1,k} \\
	1-\delta_{1,2} & 1-n_2 & \cdots & 1-\delta_{2,k} \\
	\vdots & \vdots & \ddots & \vdots \\
	1-\delta_{1,k} & 1-\delta_{2,k} & \cdots & 1-n_k\\
	 \end{array}  \right]$$
	 where $n_j$ for $j=1,\cdots, k$ denotes the number of multi-intersection points on $L_j$ and 
	 $$\delta_{i,j}=\begin{cases} 0 & \text{if } L_i\cap L_j \text{ is a double point}\\ 1 & \text{if } L_i\cap L_j \text{ is a multi-intersection point} \end{cases}.$$
	
	We show we can take $z\in (0,\infty)^{k+N}$ to be the vector whose coordinates are all $1$, to verify the positive G-S condition. Observe that $n_j$, the number of multi-intersection points on $L_j$ can be calculated as
	$$n_j = e_j^T A \left[ \begin{array}{c} 1\\ \vdots \\ 1 \end{array} \right]$$
	where $e_j^T$ is the $1\times k$ matrix with a $1$ in the $j^{th}$ column and $0$ elsewhere and the last vector is a $N\times 1$ vector of all $1$'s. Similarly, we can write $m_p$, the number of lines through the $p^{th}$ multi-point as
	$$m_p = \left[ 1 \cdots 1 \right] A e_p.$$
	
	Then the first $k$ coordinates of $Qz$ are
	$$(Qz)_j= 1-n_j+\sum_{i\neq j}(1-\delta_{i,j})+n_j \geq 1$$
	and the last $N$ coordinates are
	$$(Qz)_{k+p}= m_p-1\geq 2$$
	(since every multi-intersection point lies on at least $3$ lines).
	
	Therefore by \cite{limak:concavenbhd} the plumbing $P_\A$ supports a symplectic structure where the core spheres of the plumbing are symplectic, they intersect $\omega$-orthogonally, and there is an inward pointing Liouville vector field demonstrating that the boundary is concave.
	
	To build $N_\A$, blow-down the exceptional $(-1)$ spheres in $P_\A$. To do this symplectically in a controlled manner, choose an almost complex structure $J$ tamed by $\omega$ for which the core symplectic spheres (particularly the $-1$ spheres) are $J$-holomorphic and then blow-down those $-1$ spheres \cite{mcduff:rationalruled}.
\end{proof}

The symplectic structure on the plumbing from \cite{gaystipsicz:convexnbhds, limak:concavenbhd} is obtained by gluing together symplectic structures on pieces. For each disk bundle to be plumbed in $n$ places, Gay and Stipsicz associate a trivial disk bundle over the surface (which in our case is always a sphere) with $n$ holes: $(S^2 \setminus (D_1\sqcup \cdots \sqcup D_n))\times D^2$. The symplectic structure is $\beta_v+rdr\wedge d\theta$ with Liouville vector field $V_v=W_v+\left(\frac{1}{2}r-\frac{z_v}{r}\right)\partial_r$. Here $\beta_v$ is a symplectic form on $\Sigma_v = (S^2\setminus (D_1 \sqcup \cdots \sqcup D_n))$ with a Liouville vector field $W_v$ (which one can choose to point inward or outward in a standard way near each boundary component). For each plumbing to be performed, Gay and Stipsicz construct a toric model for the transversally intersecting disks whose moment map image is shown in figure \ref{fig:elbow}. Li and Mak observe that if the plumbing satisfies the positive G-S criterion, then the vector $z$ specifies a translation of this toric base in $\R^2$ into the third quadrant so that the standard Liouville vector field in toric coordinates points inwards on the boundary. Moreover these toric pieces fit together with the disk bundle pieces so that the symplectic and Liouville structures align. 

Note that because the Liouville vector field is defined and non-zero away from the core symplectic spheres, by Moser's theorem, this concave structure exists in arbitrarily small neighborhoods of any collection of symplectic spheres intersecting according to given the plumbing graph. By blowing up and then blowing down on the interior of a neighborhood of a line arrangement, one sees that such concave neighborhoods exist in any small neighborhood of a symplectic realization of a line arrangement. Because of this property, we will call the concavely induced contact manifold on the boundary the \emph{canonical contact manifold for the line arrangement} $(Y_\A,\xi_\A)$.

\begin{figure}
	\centering
	\includegraphics[scale=.5]{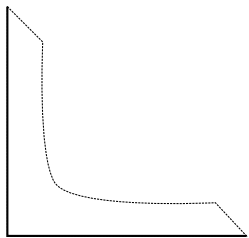}
	\caption{Toric base for a neighborhood of a plumbing region.}
	\label{fig:elbow}
\end{figure}

The contact structure induced by this plumbing is a particularly natural contact structure on a graph manifold and agrees with the canonical contact structures on lens spaces and Seifert fibered spaces. Observe that the Reeb vector field on the contact boundary is a multiple of $\partial_{\theta}$ away from the plumbing regions so the Reeb flow is just given by the circle fibration. Near the plumbing regions given by figure \ref{fig:elbow}, the boundary piece is a $T^2\times I$. The contact planes here are transverse to each $T^2\times\{s\}$ and the characteristic foliations on these tori are linear with slopes varying monotonically as $s$ varies in $I$. Such tori in a contact manifold are called \emph{pre-Lagrangian}. A result of Colin \cite{colin:tightcontact} says that if $(Y,\xi)$ is a contact manifold containing an incompressible pre-Lagrangian torus $T$ then if the induced contact structure on $Y\setminus T$ is universally tight, then the glued up contact manifold $(Y,\xi)$ is universally tight. Using this result we can prove the following.

\begin{proposition}
	The contact manifold $(Y_\A,\xi_\A)$ coming from any line arrangement $\mathcal{A}$ is (universally) tight.
\end{proposition}

\begin{proof}
	If the line arrangement is a pencil (meaning every line intersects at a single point), then it can be realized in $\cptwo$ and has a concave neighborhood inducing the canonical contact structure on the boundary. The complement is a symplectic filling of this contact boundary with the diffeomorphism type of a 1-handlebody filling $\conn_n S^1\times S^2$. Therefore the canonical contact structure is fillable and therefore tight and is the unique tight contact structure on $\conn_n S^1 \times S^2$.
	
	For any other line arrangement, the corresponding plumbing $P_\A$ has the property that every disk bundle is plumbed in at least two places. In this case, the pre-Lagrangian torus associated to each plumbing is incompressible in the 3-manifold. This is because it is incompressible in the $S^1\times \Sigma_v$ piece on the boundary of the incident disk bundles and incompressibility is preserved under gluing by the Seifert van Kampen theorem and the fact that a group injects into its amalgamated product with another group or any HNN extension. 
	
	The contact structure on a $S^1\times \Sigma_v$ piece on the boundary of a trivial disk bundle is induced by the symplectic structure and Liouville vector field on $D^2\times \Sigma_v$. The contact form is
	$$\alpha_v = \eta_v+\frac{1}{2}r^2d\theta-z_vd\theta$$
	where $\eta_v$ is a $1$-form on $\Sigma_v$. The Lie derivative in the $S^1$ fiber direction of this form is
	$$\mathcal{L}_{\partial_\theta}(\eta_v+\frac{1}{2}r^2d\theta-z_vd\theta)=\iota_{\partial_\theta}dr\wedge d\theta+d(\frac{1}{2}r^2)=-dr+dr=0.$$
	Therefore the contact structure is $S^1$ invariant on $\Sigma\times S^1$ and such contact structures are universally tight. 
	
	By repeatedly using Colin's theorem from \cite{colin:tightcontact} for each plumbing, we conclude that the glued up contact manifold is universally tight.
\end{proof}

\begin{remark}
	Note that this argument applies to plumbings more general than those corresponding to line arrangements. For any plumbing satisfying the positive G-S criterion where each disk bundle is plumbed in at least two places, there is a concavely induced universally tight contact structure on the boundary. When there are disk bundles plumbed in only one place, the torus associated to that plumbing is no longer incompressible and in certain cases (but not generically), the resulting contact structure is overtwisted (for example on the $S^3$ boundary of a plumbing of two spheres of self-intersection numbers $1$ and $2$).
\end{remark}

We now prove Theorem~\ref{T:non-filling} of the introduction, which we restate for the reader's convenience.
\begin{T:non-filling}
	Suppose a combinatorial line arrangement $\A$ is not symplectically realizable in $\cptwo$. Then the canonical contact manifold $(Y_\A,\xi_\A)$ is not strongly symplectically fillable.
\end{T:non-filling}

Of course, any arrangement $\A$ which is not topologically realizable in $\cptwo$ is not symplectically realizable, so each of the arrangements whose topological realizability we can obstruct gives rise to a non-strongly-fillable, universally tight contact manifold.
\begin{proof}
	Suppose $(Y_\A,\xi_\A)$ had a strong symplectic filling $(W,\omega)$. Then after rescaling and adding a collar which is a subset of the symplectization of $(Y_\A,\xi_\A)$ to $(W,\omega)$ so that the boundary contact forms match up, we can glue (the slightly modified) $(W,\omega)$ to the concave plumbing $P_\A$ to obtain a closed symplectic manifold $(X,\omega)$. The core spheres of the plumbing $P_\A$ are symplectic and $\omega$-orthogonal so there is an almost complex structure $J$ for which all of these spheres are $J$-holomorphic. Blow-down the $J$-holomorphic $-1$-spheres of the plumbing. Then the other spheres of the plumbing descend to $J$-holomorphic spheres of square $+1$ which are the core of the concave neighborhood $N_\A$. By a theorem of McDuff \cite{mcduff:rationalruled}, $(X,\omega)$ is symplectomorphic to a blow-up of $\cptwo$ and one of the $+1$-spheres is identified with the complex projective line under this symplectomorphism. The other symplectic $+1$ spheres are $J'$-holomorphic under the almost complex structure $J'$ on the blow-up of $\cptwo$ obtained by pushing forward $J$ on $X$ by the symplectomorphism. Their homology classes in $\cptwo \conn M\cptwobar$ have the form $a_0 \ell +\sum_{i=1}^Ma_ie_i$ where $a_0=1$ because each of these spheres algebraically intersects the sphere identified with $\cpone$ exactly once. The other $a_i$ are necessarily zero because $1=[L_j]^2 = 1-\sum a_i^2$. Therefore, the homological intersection of any line with any exceptional sphere is zero. All of the lines are $J'$-holomorphic. If the number of blow-ups, $M$, is greater than zero, we can find a $J'$-holomorphic exceptional sphere to blow-down, and it is necessarily disjoint from the $J'$-holomorphic spheres in $N_\A$ because its homological intersection with them is zero. Therefore we can blow-down exceptional spheres disjoint from $N_\A$ until we reach $\cptwo$. This gives a symplectic realization of the line arrangement $\A$ in $\cptwo$ contradicting the assumption.
\end{proof}

\begin{remark}
	In many cases, one can show that the contact manifolds arising as concave boundaries of plumbings of spheres are non-fillable by a more direct application of McDuff's theorem. For example, if there are two positive spheres in the plumbing which algebraically span a rank $2$ subspace of the positive summand of second homology, then $b_2^+$ any manifold containing this plumbing will be at least $2$. Since McDuff's theorem implies that any closed symplectic manifold containing a positive symplectic sphere has $b_2^+=1$, this means that the concave plumbing cannot be filled. The line arrangement case is a borderline case where this condition fails. The concave cap $N_\A$ has many positive symplectic spheres, but they algebraically only have rank one in second homology, so the fillability question is more subtle. Indeed, it is equivalent to the much more subtle question of symplectic realizability of the line arrangement.
\end{remark}

\bibliography{fano}
\bibliographystyle{amsplain}

\end{document}